\documentclass[11pt,reqno,a4paper]{amsart}
\usepackage{amsmath,amsthm,amsfonts,amssymb,xcolor,enumitem,comment,hyperref}
\usepackage[numbers,sort&compress]{natbib}
\usepackage{graphicx}
\usepackage{setspace}
\usepackage{mathrsfs}
\usepackage{hyperref}
\usepackage[letterpaper,margin=1.3in]{geometry}

\newcommand{\R}{\mathbb{R}}

\renewcommand{\H}{\mathbb{H}}

\newcommand{\g}{\gamma}

\renewcommand{\epsilon}{\varepsilon}
\renewcommand{\hat}{\widehat}
\renewcommand{\tilde}{\widetilde}

\newtheorem{theorem}{Theorem}[section]
\newtheorem{proposition}[theorem]{Proposition}
\newtheorem{lemma}[theorem]{Lemma}

\newtheorem*{question}{Question}

\theoremstyle{definition}
\newtheorem{definition}[theorem]{Definition}

\theoremstyle{remark}

\numberwithin{equation}{section}

\newcommand{\diam}{\mathop\mathrm{diam}\nolimits}

\title[The finiteness principle for curves in the Heisenberg group]{Whitney's Extension Theorem and the finiteness principle for curves in the Heisenberg group}
\author[Scott Zimmerman]{Scott Zimmerman}
\address{Department of Mathematics, 
The Ohio State University at Marion,
1465 Mt Vernon Ave, Marion, OH 43302, United States}
\email{zimmerman.416@osu.edu}

\subjclass{58C25, 53C17.}

\allowdisplaybreaks

\begin{document}

\begin{abstract}
Consider the sub-Riemannian Heisenberg group $\H$.
In this paper, we answer the following question:
given a compact set $K \subseteq \R$ and a continuous map $f:K \to \H$,
when is there a horizontal $C^m$ curve $F:\R \to \H$ such that
$F|_K = f$?
Whitney originally answered this question for real valued mappings, 
and
Fefferman provided a complete answer for real valued functions defined on subsets of $\R^n$. 
We also prove a finiteness principle
for $C^{m,\sqrt{\omega}}$ horizontal curves in the Heisenberg group in the sense of Brudnyi and Shvartsman. 
\end{abstract}

\keywords{Finiteness principle; Whitney extension theorem; Heisenberg group}

\date{\today}

\maketitle

\tableofcontents

\newpage

\section{Introduction}

Fix a compact set $K \subseteq \mathbb{R}^n$ and a continuous function $f:K \to \R$,
and consider the following question:

\smallskip

\noindent
\textbf{Whitney's question.} When is there some $F \in C^m(\R^n)$ such that $F|_K = f$?

\smallskip

In other words, when do we have $f \in C^m(K)$?
(For all definitions of the relevant function and trace spaces, see Subsection~\ref{sec-function}.)
In his classical extension theorem \cite{Whitney},
Whitney proved that $f \in C^m(K)$ if and only if
there is a collection
of functions $( f_\alpha )_{|\alpha| \leq m}$
which act as the partial derivatives of $f = f_0$ in the sense of Taylor's theorem.
Moreover, the extension $F$ can be chosen in such a way that $\partial^\alpha F = f_\alpha$ on $K$.
See Theorem~\ref{t-WhitClassic} below for the statement of Whitney's result in $\R$.

A $C^1$ version of this result was proven in \cite{FraSerSer}
for real valued mappings defined on subsets of the 
sub-Riemannian Heisenberg group $\H$.
One naturally then considers the problem of smoothly extending
a map from a subset of Euclidean space {\em into} $\H$.
In this setting, however, one also requires that the extension 
respects
the sub-Riemannian geometry of $\H$.
As noted in Proposition~4.1 from \cite{ZimSpePinWhitney},
Whitney's classical assumptions 
do not suffice to guarantee 
the existence of such an extension in this setting,
and, as such, more assumptions on $f$ are required.

A version of
Whitney's classical extension theorem 
from \cite{Whitney}
was proven by the author 
for horizontal $C^1$ curves in the Heisenberg group \cite{ZimWhitney}
and by Pinamonti, Speight, and the author for horizontal $C^m$ curves \cite{ZimSpePinWhitney}.
See Theorem~\ref{t-HeisWhit} for the statement of and more discussion regarding these results.
This extension theorem has since been applied to verify the existence of 
Lusin-type approximations of Lipschitz curves in the Heisenberg group by horizontal $C^m$ curves \cite{PinGarCap}.
Whitney-type extensions for horizontal $C^1$
curves 
in general Carnot groups and sub-Riemannian manifolds
have been considered by Julliet, Sacchelli, and Sigalotti \cite{Pliability,SaccSiga}.

Let us return to Whitney's question.
We would like an answer
without having first assigned a family of
``derivatives'' of $f$ on $K$.
In the case $n=1$, Whitney provided an answer
using the divided differences of $f$.
\begin{theorem}[\cite{Whitney2}]
\label{t-WhitFin}
Suppose $K \subseteq \mathbb{R}$ is compact
and $f:K \to \R$ is continuous.
Then $f \in C^m(K)$ if and only if 
the $m$th divided differences 
of $f$ converge uniformly on $K$.
\end{theorem}
See Subsection~\ref{s-DD} for more information about divided differences.
Glaeser later provided an answer to Whitney's question for functions in $C^1(\R^n)$
using a geometric argument \cite{GlaeserWhit}.
The proof of Theorem~\ref{t-WhitFin} actually implies the following.
\begin{theorem}[\cite{Whitney2}]
\label{t-WhitCheat}
Suppose $K \subseteq \R$ is compact.
Then there is a bounded, linear operator $W:C^m(K) \to C^m(\R)$ such that $Wf|_K = f$ for all $f \in C^m(K)$.
\end{theorem}
Brudnyi and Shvartsman \cite{BruShv,BruShv4}
observed the following reformulation of Whitney's result which has since come to be known as the {\em finiteness principle}.
\begin{theorem}[\cite{Whitney2}]
\label{t-brush}
Suppose $K \subseteq \R$ is compact and $f:K \to \R$ is continuous.
There is some $F \in C^{m,\omega}(\R)$ with $F = f$ on $K$
if and only if there is a constant $M>0$ such that,
for every subset $X \subseteq K$ with $\# X = m+2$,
there is some $F_X \in C^{m,\omega}(\R)$ such that 
$F_X = f$ on $X$ and $\Vert F_X \Vert_{C^{m,\omega}(\R)} \leq M$.
\end{theorem}
The statement of Theorem~\ref{t-brush} does not involve divided differences
and allows one to consider Whitney's theorem for higher dimensional domains.
As a result, Brudnyi and Shvartsman generalized the finiteness principle to $C^{1,\omega}(\R^n)$
functions in \cite{BruShv}.
Their continued work on this problem
can be found in 
\cite{BruShv2,BruShv,BruShv5,BruShv6,Brud1,BruShv3,Shv1,Shv2,Shv3}.
In \cite{BMP1,BMP2}, Bierstone, Milman, and Paw\l{}ucki
considered Whitney's question for extensions from subanalytic sets in $\R^n$,
and Fefferman answered Whitney's question fully in \cite{FefferWhitFull,FefferWhitLinear}.
He also proved versions of the finiteness principle for $C^{m,\omega}(\R^n)$ functions 
\cite{FefferWhit,FefferWhit2}.
Recent updates on this project 
by Fefferman, Israel, and Luli
can be found in \cite{FefIsrLul,FefIsrLul2}
and by 
Carruth, Frei-Pearson, Israel, and Klartag in \cite{CoordFree}.
An extensive history of work related to Whitney's question from the past few decades can be found in \cite{FefferSummary}.

In this paper, we will focus on mappings into
the sub-Riemannian Heisenberg group $\H^n$
and consider a Heisenberg-version of Whitney's question.
(See Subsection~\ref{s-Heis} for the appropriate definitions.)
Suppose $K \subseteq \R^n$ is compact,
and fix $f:K \to \R^{2k+1}$.

\smallskip

\noindent
\textbf{Whitney's question in $\H$.} When is there a map $F \in C^m (\R^n,\R^{2k+1})$ such that $F|_K = f$
and $F$ is horizontal?

\smallskip

For the purposes of this paper, we will consider only the setting $\H := \H^1 = \R^3$.
However, all results discussed below hold in higher dimensional Heisenberg groups
with the appropriate changes in notation.

As mentioned above, 
Whitney's question in $\H$ was answered on subsets of $\R$
in \cite{ZimSpePinWhitney} and \cite{ZimWhitney} 
when the derivatives of the extension
are required to have prescribed values on $K$.
This is an analogue of Whitney's original result from \cite{Whitney}.
Now, we will provide an answer to Whitney's question in $\H$
in the case $n=1$
in analogy to Theorem~\ref{t-WhitFin}.

The following are the main results of this paper.
For a compact set $K \subseteq \R$ and a map $\g:K \to \H$,
we will write $\g \in C_{\mathbb{H}}^{m}(K)$ to indicate that
there is a horizontal curve $\Gamma \in C^m(\R,\R^3)$
with $\Gamma|_K = \g$.
\begin{theorem}
\label{c-m1}
Assume $K \subseteq \mathbb{R}$ is compact, 
and suppose $\g:K \to \H$ is continuous.
Then $\g \in C_{\mathbb{H}}^1(K)$ if and only if
the Pansu difference quotients of $\gamma$ converge uniformly on $K$ to horizontal points.
\end{theorem}
See Definition~\ref{d-deriv} for a discussion of the Pansu difference quotients.
For higher order derivatives, we have the following.

\begin{theorem}
\label{t-supermain}
Assume $K \subseteq \mathbb{R}$ is compact with finitely many isolated points
and $\g:K \to \H$ is continuous.
Then $\g \in C_{\mathbb{H}}^m(K)$ if and only if
\begin{enumerate}
\item
the $m$th divided differences of $\gamma$ converge uniformly on $K$,
\item $\gamma$ satisfies the discrete $A/V$ condition on $K$.
\end{enumerate}
\end{theorem}

Condition {\em (1)} is clearly necessary due to Whitney's result (Theorem~\ref{t-WhitFin}).
The discrete $A/V$ condition in {\em (2)} is an analogue of 
the $A/V$ condition introduced in \cite{ZimSpePinWhitney},
and both are generalizations of the Pansu difference quotient. 
See Sections~\ref{s-AV} and \ref{s-dAV} for a thorough discussion of these conditions.
According to Proposition~5.2 in \cite{ZimWhitney},
the $A/V$ condition is
necessary when extending to smooth, horizontal curves in $\H$.
While the $A/V$ condition from \cite{ZimSpePinWhitney} relies on information from a collection of functions defined on $K$,
the discrete $A/V$ condition in {\em(2)} above requires knowledge only 
of the values of $\g$.
The definition of the discrete $A/V$ condition replaces Taylor polynomials with interpolating polynomials just as Whitney did in his classical proofs.

The following result holds for arbitrary compact sets $K \subseteq \R$.
For $\gamma=(f,g,h)$ with $f,g,h \in C^m(K)$,
we will write $W\g$ to denote the curve $(Wf,Wg,Wh) \in C^m(\R,\R^3)$
where $W$ is the linear operator whose existence is guaranteed by Theorem~\ref{t-WhitCheat}.
\begin{theorem}
\label{t-supermainWf}
Assume $K \subseteq \mathbb{R}$ is compact and $\g:K \to \H$ is continuous.
Then $\g \in C_{\mathbb{H}}^m(K)$ if and only if
\begin{enumerate}
\item
the $m$th divided differences of $\gamma$ converge uniformly on $K$,
\item $W\gamma$ satisfies the $A/V$ condition on $K$.
\end{enumerate}
\end{theorem}
The advantage of this result over Theorem~\ref{t-supermain} is clearly in its generality for all compact sets.
However, one might expect that the hypotheses of Theorem~\ref{t-supermainWf}
are harder to ``compute'' (in the sense of \cite{FittingI,FittingII,FittingIII})
than those of Theorem~\ref{t-supermain}.
This is summarized in the following (imprecise) question:
\begin{question}
Suppose $K \subseteq \R$ is compact,
$\g:K \to \H$ is continuous,
and the $m$th divided differences of $\g$ converge uniformly on $K$. 
Which is easier: verifying that $\g$ satisfies the discrete $A/V$ condition on $K$
or computing $W\g$ and verifying that it satisfies the $A/V$ condition on $K$?
\end{question}

Finally, we come to our discussion of the finiteness principle (Theorem~\ref{t-brush})
for curves in the Heisenberg group.
Just as in the Euclidean setting, 
the statement of the following result removes all mention of divided differences.

\begin{theorem}
\label{t-finiteness}
Assume $K \subseteq \mathbb{R}$ is compact
with finitely many isolated points
and $\# K \geq m+2$
for some positive integer $m$,
and suppose $\g:K \to \H$ is continuous.
If
there exist a 
modulus of continuity $\omega$ and a
constant $M > 0$ such that,
for any $X \subseteq K$ with $\#X = m+2$,
there is a curve $\Gamma_X \in C^{m,\omega}(\R,\R^3)$
with $\Gamma_X = \gamma$ on $X$,
$
\Vert \Gamma_X \Vert_{C^{m,\omega}(\R,\R^3)}  \leq M,
$
and
$$
\left| \frac{A(\Gamma_X;a,b)}{V(\Gamma_X;a,b)} \right| \leq M\omega(b-a)
\quad \text{for all } a,b \in K \text{ with } a<b,
$$
then there is a horizontal curve $\Gamma \in C^{m,\sqrt{\omega}}(\R,\R^3)$
such that $\Gamma|_K = \g$.
\end{theorem}
Note the drop in regularity of the $m$th derivative.
This is a result of the construction of the horizontal $C^m$ extension theorem in \cite{ZimSpePinWhitney}.
The following proposition hints that this drop is due to the construction itself and may be possible to remedy.
\begin{proposition}
\label{p-finiteness}
If $\g \in C^{m,\omega}(\R,\R^3)$ and $\g$ is horizontal, then
there is a constant $M > 0$ such that
$$
\left| \frac{A(\g;a,b)}{V(\g;a,b)} \right| \leq M\omega(b-a)
\quad \text{for all } a,b \in \R \text{ with } a<b.
$$
\end{proposition}
The proof of this proposition is nearly identical to that of Proposition~5.1 in \cite{ZimSpePinWhitney}, so it will not be included below. Simply replace all instances of $\varepsilon$ with a constant multiple of $\omega(b-a)$ throughout the proof.

The paper is organized as follows.
Section~\ref{s-prelim} establishes preliminary facts about the sub-Riemannian Heisenberg group,
prior Whitney-type results, and divided differences
which will be important for the later discussion.
The bulk of the new content is contained in Sections~\ref{s-AV} and \ref{s-dAV}
wherein the $A/V$ conditions are defined and several important lemmas relating the $A/V$ condition from \cite{ZimSpePinWhitney} to the discrete $A/V$ condition are established.
Using the technical tools provided in these sections,
we then prove Theorems~\ref{c-m1}, \ref{t-supermain}, \ref{t-supermainWf}, and \ref{t-finiteness} in Section~\ref{s-proofs}.

I would like to sincerely thank the referee for their careful reading of the manuscript and for their suggestions and corrections which led to an improvement of this paper.

\section{Preliminaries}
\label{s-prelim}
Throughout the rest of the paper, $m$ will represent a positive integer, and $\omega$ will be a modulus of continuity i.e. a continuous, increasing, concave function $\omega:[0,\infty] \to [0,\infty]$ with $\omega(0)=0$.
In what follows, given an object $d$ we will write 
$a \lesssim_d b$ to indicate that $a\leq Cb$ where $C=C(d)>0$ is a constant depending possibly on $d$.
Moreover, for any integer $k>0$, we say that a nonnegative quantity $c(a,b)$ is uniformly $o(|b-a|^k)$ on a set $K$
if, for every $\epsilon > 0$, there is a $\delta>0$ such that $c(a,b) \leq \varepsilon |b-a|^k$ whenever $a,b \in K$ satisfy $|b-a| < \delta$.

\subsection{The Heisenberg group}
\label{s-Heis}
For any positive integer $n$,
we define the {\em nth sub-Riemannian Heisenberg group}
to be
$\mathbb{H}^n = \mathbb{R}^{2n+1}$
with the group law
\begin{align*}
(x,y,z)*(x',y',z')
=
\left(x+x',y+y',z+z'+2\sum_{j=1}^n(y_jx_j'-x_jy_j')\right)
\end{align*}
for $x,y,x',y' \in \R^n$
and $z,z' \in \R$.
One may check that $(x,y,z)^{-1} = (-x,-y,-z)$.
With this group law, $\H^n$ is a Lie group with left invariant vector fields
$$
X_i(p) = \tfrac{\partial}{\partial x_i} + 2y_i \tfrac{\partial}{\partial z}, \quad
Y_i(p) = \tfrac{\partial}{\partial y_i} - 2x_i \tfrac{\partial}{\partial z}, \quad
Z(p) = \tfrac{\partial}{\partial z}
\qquad
\text{for } 1 \leq i \leq n
$$
for $p = (x,y,z) \in \R^n \times \R^n \times \R$.
Since $[X_i,Y_i] = -4Z$ for each $1 \leq i \leq n$,
the Lie group $\H^n$ is a Carnot group of step 2
with horizontal distribution $\text{span} \{ X_1,Y_1,\dots,X_n,Y_n\}$. 
We say that a point in $\H^n$ is {\em horizontal} if it lies in $\mathbb{R}^{2n} \times \{0 \}$,
and an absolutely continuous curve $\gamma:\R \to \R^{2n+1}$
is {\em horizontal} if $\gamma'(t) \in \text{span} \{ X_1,Y_1,\dots,X_n,Y_n\}$
for almost every $t \in \R$.
We may equivalently write the following.
\begin{proposition}
Suppose $\gamma=(f,g,h):\R \to \R^n \times \R^n \times \R$ is absolutely continuous.
Then
$\g$ is horizontal if and only if
$$
h' = 2 \sum_{j=1}^n \left(f'_jg_j - f_jg'_j \right)
\qquad
\text{a.e. in } \R.
$$
\end{proposition}
For a proof of this, see Lemma~2.3 in \cite{GarethLusin}.
In  $\H = \H^1$,
this equation is simply $h' = 2(f'g-fg')$.
If $\g \in C^m_\H (\R)$
(i.e. $\gamma \in C^m(\R,\R^3)$ and $\g$ is a horizontal curve),
then, according to the Leibniz rule, we have
\begin{equation}
    \label{e-LeibnizRule}
h^{(k)} = 2 \sum_{i=0}^{k-1} \binom{k-1}{i} \left(f^{(k-i)}g^{(i)} - g^{(k-i)}f^{(i)} \right)
\qquad \text{for } 1 \leq k \leq m \text{ on } \R.
\end{equation}

The {\em dilations} $(\delta_r)_{r\in \R \setminus \{ 0 \}}$ defined as
$$
\delta_r(x,y,z) = (rx,ry,r^2z)
$$
form a 
family of group automorphisms on $\H^n$.
Recall that the {\em Pansu derivative} of $\g : \R \to \H$ at $x \in \R$
is defined as
$$
 \lim_{y \to x} \delta_{1/(y-x)} \left( \g(x)^{-1} * \g(y) \right)
$$
whenever this limit exists.
If $\gamma$ is Lipschitz, then this limit exists almost everywhere 
and converges to a horizontal point.
See, for example, Lemma~2.1.4 in \cite{MontiThesis}.
\begin{definition}
\label{d-deriv}
Suppose $K \subseteq \R$ is compact and fix $\g=(f,g,h):K \to \H$.
We say that the {\em Pansu difference quotients of $\g$ converge uniformly on $K$ to horizontal points} if
the difference quotients of $f$ and $g$ converge uniformly on $K$ (in the sense of divided differences (see Subsection~\ref{s-DD}))
and if there is a modulus of continuity $\alpha$ such that
$$
|h(b) - h(a) - 2 (f(b)g(a) - f(a)g(b))| \leq \alpha(|b-a|) |b-a|^2
$$
for all $a,b \in K$.
\end{definition}
According to the definition of the group law, this is equivalent to the assumption that the difference quotients formed by the first two coordinates of $\delta_{1/(b-a)} \left(\gamma(a)^{-1} * \gamma(b) \right)$ converge uniformly and that the third coordinate vanishes uniformly.
Indeed, the  third coordinate of $\gamma(a)^{-1} * \gamma(b)$ is
$
h(b) - h(a) - 2 (f(b)g(a) - f(a)g(b))
$.
Compare this definition to the statement of Theorem~1.7 in \cite{ZimWhitney}.

\subsection{Function spaces}
\label{sec-function}

Suppose $\mathcal{C}(\mathbb{R}^n)$ is a (semi)normed space of real valued functions on $\mathbb{R}^n$ with (semi)norm $\Vert \cdot \Vert_{\mathcal{C}(\mathbb{R}^n)}$.
For any measurable set $E \subset \mathbb{R}^n$, 
the {\em trace space} $\mathcal{C}(E)$ is defined as
$
\mathcal{C}(E) = \{ F|_E \, : \, F \in \mathcal{C}(\mathbb{R}^n) \}
$
with seminorm
$$
\Vert f \Vert_{\mathcal{C}(E)} = \inf \left\{  \Vert F \Vert_{\mathcal{C}(\R^n)} \, : \, F \in \mathcal{C}(\R^n), \, F|_E = f \right\}.
$$
We also define $\mathcal{C}(\mathbb{R},\mathbb{R}^3)$ to be the space of functions $\gamma=(f,g,h):\mathbb{R} \to \R^3$
such that $f,g,h \in \mathcal{C}(\R)$, and we endow this space with the seminorm
$$
\Vert \gamma \Vert_{\mathcal{C}(\mathbb{R},\mathbb{R}^3)} = \Vert f \Vert_{\mathcal{C}(\mathbb{R})} + \Vert g \Vert_{\mathcal{C}(\mathbb{R})} + \Vert h \Vert_{\mathcal{C}(\mathbb{R})}.
$$

For a positive integer $m$, the space 
$C^m(\R^n)$
consists of those $m$-times continuously differentiable functions
$f:\R^n \to \R$ such that the following seminorm is finite:
$$
\Vert f \Vert_{C^m(\R^n)} = \sup_{x \in \R^n} \sum_{|\alpha|=m} |\partial^\alpha f(x)|.
$$

If 
$\omega:[0,\infty) \to [0,\infty)$ is a modulus of continuity (i.e. an increasing, concave, continuous function with $\omega(0)=0$),
define $C^{m,\omega}(\R^n)$ to be the subspace of $C^{m}(\mathbb{R}^n)$
consisting of those functions such that the following seminorm is finite:
$$
\Vert f \Vert_{C^{m,\omega}(\R^n)} =  \sup_{\substack{x,y \in \R^n \\ x \neq y}} 
\sum_{|\alpha|=m}
\frac{\left| \partial^\alpha f(x) - \partial^\alpha f(y) \right|}{\omega(|x-y|)}.
$$
Of course, when $n=1$, we have
$$
\Vert f \Vert_{C^m(\R)} = \sup_{x \in \R} |f^{(m)}(x)|
\qquad
\text{and}
\qquad
\Vert f \Vert_{C^{m,\omega}(\R)} = \sup_{\substack{x,y \in \R \\ x \neq y}} 
\frac{\left| f^{(m)}(x) - f^{(m)}(y) \right|}{\omega(|x-y|)}
$$
for any $f:\R \to \R$.
In other words, $f \in C^{m,\omega}(\R)$
means $f^{(m)}$ is uniformly continuous with modulus of continuity $\Vert f \Vert_{C^{m,\omega}(\R)} \cdot \omega$.


\subsection{Prior Whitney-type results}
\label{s-WhitHist}
Given a collection $F=(F^k)_{k=0}^m$ of continuous, real valued functions on a set $E \subset \R$,
define the $m$th order Taylor polynomial $T_a^m F$ of $F$ at $a \in E$ as
\begin{equation}
\label{e-TaylorJet}
T_a^m F(x) 
= 
\sum_{k=0}^m \frac{F^{k}(a)}{k!}(x-a)^k
\quad
\text{for all } x \in \R.
\end{equation}
If $f:\R \to \R$ is $m$ times differentiable at $a$,
the Taylor polynomial $T_a^m f$
is defined as usual using the collection $F=(f^{(k)})_{k=0}^m$ in \eqref{e-TaylorJet}.
We will often drop the exponent and write $T_a F$ or $T_a f$ when the order of the polynomial is obvious from the context.

If $f \in C^m(\R)$
and $K \subseteq \R$ is compact,
then, 
according to Taylor's theorem,
\begin{equation}
    \label{e-taylorApprox}
    \lim_{\substack{|b-a| \to 0 \\ a,b \in K}}
    \frac{\left| f^{(k)}(b) - T_a^{m-k}f^{(k)}(b) \right|}{|b-a|^{m-k}} = 0
    \qquad
    \text{for } 0 \leq k \leq m.
\end{equation}
We note in particular that there is a modulus of continuity $\alpha$ with
\begin{equation}
\label{e-taylor0}
|f^{(m)}(x)-f^{(m)}(y)| \leq \alpha(|x-y|),
\end{equation}
\begin{equation}
\label{e-taylor1}
|f(y)-T_x^mf(y)| \leq \alpha(|x-y|) |x-y|^m,
\end{equation}
\begin{equation}
\label{e-taylor2}
|f'(y)-(T_x^mf)'(y)| \leq \alpha(|x-y|) |x-y|^{m-1}
\end{equation}
for all $x,y \in K$
since $(T_x^mf)' = T_x^{m-1}(f')$.

We call a collection $F=(F^k)_{k=0}^m$ of continuous, real valued functions on a closed set $K \subset \R$
a {\em Whitney field of class $C^m$ on $K$}
if
$$
\lim_{\substack{|b-a| \to 0 \\ a,b \in K}} 
\frac{\left|F^k(b) - T_a^{m-k} F^k(b) \right|}{|b-a|^{m-k}} = 0
    \qquad
    \text{for } 0 \leq k \leq m.
$$
where $T_a^{m-k} F^k$ is the $(m-k)$th order Taylor polynomial of the collection $(F^j)_{j=k}^{m}$.
Note that, 
if $f \in C^m(\R)$, then the collection $F= (f^{(k)})_{k=0}^m$ is a Whitney field of class $C^m$ on $K$ for any compact $K \subseteq \R$.
From Whitney's classical extension theorem in dimension 1, we have the following
\begin{theorem}[\cite{Whitney}]
\label{t-WhitClassic}
Suppose $F=(F^k)_{k=0}^m$ is a collection of continuous, real valued functions on a compact set $K \subset \R$.
There is a function
$f \in C^m(\R)$ satisfying $f^{(k)}|_K = F^k$ for $0 \leq k \leq m$
if and only if
$F$ is a Whitney field of class $C^m$ on $K$.
\end{theorem}

See \cite{Bierstone} for a proof.
Suppose now that $f \in C^{m,\omega}(\R)$.
According to (2) in \cite{FollandRemainder},
there is a constant $C>0$ such that 
\begin{align}
\label{e-Taylor-int}
\frac{|f^{(k)}(b) - T_a^{m-k}f^{(k)}(b)|}{|b-a|^{m-k}}
\leq C \omega(|b-a|)
\end{align}
for any $a,b \in \R$ and $0 \leq k \leq m$.
For a closed set $K \subseteq \R$,
a collection $F=(F^k)_{k=0}^m$ of continuous, real valued functions on $K$ is a {\em Whitney field of class $C^{m,\omega}$ on $K$}
if there is a constant $C >0$ such that
$$
\frac{\left|F^k(b) - T_a^{m-k} F^k(b) \right|}{|b-a|^{m-k}} \leq C \omega(|b-a|)
    \qquad
    \text{for } a,b \in K, \, 0 \leq k \leq m.
$$

A proof similar to that of Theorem~\ref{t-WhitClassic}
then implies the following result.
This theorem will be useful when proving the finiteness principle (Theorem~\ref{t-finiteness}).
\begin{theorem}[\cite{Whitney}]
\label{t-WhitClassicLip}
Suppose $F=(F^k)_{k=0}^m$ is a collection of continuous, real valued functions on a compact set $K \subset \R$.
There is a function
$f \in C^{m,\omega}(\R)$ satisfying $f^{(k)}|_K = F^k$ for $0 \leq k \leq m$
if and only if
$F$ is a Whitney field of class $C^{m,\omega}$ on $K$.
\end{theorem}

The following version of Theorem~\ref{t-WhitClassic}
was proven for $C^1$ horizontal curves in the Heisenberg group in \cite{ZimWhitney}
and for $C^m$ horizontal curves in \cite{ZimSpePinWhitney}.

\begin{theorem}[\cite{ZimSpePinWhitney,ZimWhitney}]
\label{t-HeisWhit}
Suppose $K \subseteq \R$ is compact and 
$F=(F^k)_{k=0}^m$, $G=(G^k)_{k=0}^m$, and $H=(H^k)_{k=0}^m$ are collections of continuous, real valued functions on $K$.
There is a curve $\Gamma \in C^m_\H(\R)$  
satisfying $\Gamma^{(k)}|_K = (F^k,G^k,H^k)$
for $0 \leq k \leq m$
if and only if 
\begin{enumerate}
    \item 
    \label{c-whitfield}
    $F$, $G$, and $H$ are Whitney fields of class $C^m$ on $K$,
    \item \label{c-leibniz} for every $1 \leq k \leq m$ and $x \in K$, the following holds on $K$:
    $$
        H^k = 2 \sum_{i=0}^{k-1} \binom{k-1}{i} \left(F^{k-i}G^i- G^{k-i}F^i \right),
    $$
    \item \label{c-av} $(F^0,G^0,H^0)$ satisfies the $A/V$ condition on $K$.
\end{enumerate}
\end{theorem}

Condition {\em (1)} here was discussed above, and condition {\em (2)} is a consequence of the Leibniz rule as in \eqref{e-LeibnizRule}.
Condition {\em (3)} (discussed at length in Section~\ref{s-AV})
establishes a control on the rate at which the curve gathers symplectic area in the plane,
and this area is fundamentally tied to the height of a horizontal curve in the Heisenberg group. 
See \cite{ZimSpePinWhitney, GarethLusin, ZimWhitney} for more discussion on this relationship.

We will also record one direction of this result 
for $C^{m.\omega}$ curves.
The proof is similar to the one found in \cite{ZimSpePinWhitney},
and the differences are noted below.
This result is why 
Theorem~\ref{t-finiteness}
produces a curve of class $C^{m,\sqrt{\omega}}$
rather than $C^{m,\omega}$,
and it is not obvious to me how the construction in \cite{ZimSpePinWhitney} can be strengthened.
\begin{theorem}[\cite{ZimSpePinWhitney}]
\label{t-HeisWhitLip}
Suppose $K \subseteq \R$ is compact and 
$F=(F^k)_{k=0}^m$, $G=(G^k)_{k=0}^m$, and $H=(H^k)_{k=0}^m$ are collections of continuous, real valued functions on $K$.
If there is a constant $\hat{C}>0$ such that
\begin{enumerate}
    \item 
    $F$, $G$, and $H$ are Whitney fields of class $C^{m,\omega}$ on $K$,
    \item for every $1 \leq k \leq m$ and $x \in K$, the following holds on $K$:
    $$
        H^k = 2 \sum_{i=0}^{k-1}  \binom{k-1}{i}  \left(F^{k-i}G^i- G^{k-i}F^i \right),
    $$
    \item 
    and, writing $\gamma = (F_0,G_0,H_0)$,
    $$
\left| \frac{A(\g;a,b)}{V(\g;a,b)} \right| \leq \hat{C}\omega(b-a)
\quad \text{for all } a,b \in K \text{ with } a<b,
$$
\end{enumerate}
then there is a horizontal curve $\Gamma \in C^{m,\sqrt{\omega}}(\R,\R^3)$  
satisfying $\Gamma^{(k)}|_K = (F^k,G^k,H^k)$
for $0 \leq k \leq m$
\end{theorem}
\begin{proof}
This follows from the proof of Theorem~6.1 in \cite{ZimSpePinWhitney}.
We note the differences here.
Rather than invoking Whitney's original extension theorem (which is Theorem~\ref{t-WhitClassic} in this paper and Theorem~2.8 in \cite{ZimSpePinWhitney}), we instead use Theorem~\ref{t-WhitClassicLip} above to extend the Whitney fields $F$ and $G$ to $C^{m,\omega}$ functions $f$ and $g$ on $\R$. 
Moreover, it follows from the definition of the Whitney field of class $C^{m,\omega}$ 
and condition {\em(3)} above
that we may replace the modulus of continuity $\alpha$ in the estimates (2.3), (2.4), and (6.1)-(6.6)
in \cite{ZimSpePinWhitney}
with a constant multiple of $\omega$.

Since $K$ is compact, the modulus of continuity $\beta$ in Proposition~6.2 may then be replaced by a constant multiple of $\sqrt{\omega}$.
(This is because, in \cite{ZimSpePinWhitney}, $\beta$ is bounded by a constant multiple of $\sqrt{\hat{\alpha}} = \sqrt{\alpha + \alpha^2}$.
This is where the drop in regularity occurs.)
The proofs of Lemma~6.7 and Proposition~6.8 then follow.
\end{proof}

\subsection{Divided differences}
\label{s-DD}
Fix $A \subseteq \R$.
For any $f:A \to \mathbb{R}$ and
any set of $m+1$ distinct points $X = \{ x_0,\dots,x_m \}\subseteq A$, define
\begin{align}
f[x_0] &= f(x_0) \nonumber \\
f[x_0,\dots,x_k] &= \frac{f[x_1,\dots, x_k] - f[x_0,\dots,x_{k-1}]}{x_k-x_0}
\quad
\text{for } 1 \leq k \leq m. \label{e-dddefine}
\end{align}
We will call $f[X] = f[x_0,\dots,x_m]$ an {\em $m$th divided difference of $f$},
and, if $K \subseteq \R$ is compact, 
we say that the {\em $m$th divided differences of $f$ converge uniformly on K} if,
for every $\varepsilon > 0$, there is a $\delta > 0$ such that
$
|f[X] - f[Y]| <\varepsilon
$
whenever 
$X$ and $Y$ are sets of $m+1$ distinct points in $K$ and
$\diam(X \cup Y)<\delta$.
For $\gamma:K \to \R^3$,
we define $\gamma[X] := (f[X],g[X],h[X])$.
The following equivalent definition of divided differences for $C^m$ functions is Theorem~2 on page 250 of \cite{NumAnalBook}.
\begin{proposition}
\label{p-intgral}
Fix $m+1$ distinct points $x_0,\dots,x_m \in \R$.
If $f \in C^m(I)$ where $I$ is some interval containing $\{ x_0,\dots,x_m \}$, then 
\begin{align*}
    f[x_0,\dots,x_m]
    &=
    \int_0^1 \int_0^{t_1} \cdots \int_0^{t_{m-1}}
    f^{(m)} ( t_m (x_m - x_{m-1}) + \cdots \\
    & \hspace{2in} + t_1 (x_1 - x_{0}) + x_0 ) \, dt_m \cdots dt_2 dt_1.
\end{align*}
\end{proposition}
In particular,
as long as $f$ is of class $C^m$,
the map 
$(x_0,\dots,x_m) \mapsto f[x_0,\dots,x_m]$ extends to a continuous function on $I^{m+1}$,
and the recursive condition \eqref{e-dddefine}
holds for sets of not necessarily distinct
points.

\subsection{Newton interpolation polynomials}
Given a 
set $A \subseteq \R$,
a function $f:A \to \R$, and a
finite set $X = \{x_0,\dots,x_k\} \subseteq A$, 
the associated Newton interpolation polynomial is defined as
\begin{align*}
P(X;f)(x)
= f[x_0]+(x-x_0)f[x_0,x_1] 
+ \cdots+ 
(x-x_0)\cdots(x-x_{k-1})f[x_0,\dots,x_k].
\end{align*}
This is the unique polynomial of degree at most $k$ which satisfies 
$P(x_i) = f(x_i)$ for $i= 0,\dots,k$.

\begin{lemma}
\label{l-poly}
Suppose 
$\alpha$ is a modulus of continuity and
$f \in C^{m,\alpha}(I)$ for some compact interval $I$.
There is a constant $C>0$ 
depending only on $m$ and $\Vert f \Vert_{C^{m,\alpha}(I)}$
such that, for any
$X \subseteq I$
with $\#X = m+1$ 
and $P = P(X;f)$,
\begin{equation}
\label{e-poly2}   
\frac{|f(x)-P(x)|}{\diam(X)^{m}} \leq C\alpha (\diam(X))
\quad
\text{and}
\quad
\frac{|f'(x)-P'(x)|}{\diam(X)^{m-1}} \leq C\alpha (\diam(X))
\end{equation}
for all $x \in [\min X,\max X]$.
\end{lemma}
\begin{proof}
Write $M = \Vert f \Vert_{C^{m,\alpha}(I)}$.
For any $y_0,\dots,y_m,z_0,\dots,z_m \in I$, Proposition~\ref{p-intgral} gives 
\begin{align*}
    |f[y_0,\dots,y_m] &- f[z_0,\dots,z_m]|\\
    &\leq
    M  \int_0^1 \int_0^{t_1} \cdots \int_0^{t_{m-1}}
    \alpha \big( |t_m ((y_m - y_{m-1})-(z_m-z_{m-1})) + \cdots \\
    &\hspace{1in} + t_1 ((y_1 - y_{0})-(z_1-z_{0})) + (y_0 - z_0)| \big) \, dt_m \cdots dt_2 dt_1\\
    &\leq
    M  \alpha \big(|y_m-z_m| + 2|y_{m-1}-z_{m-1}| + \cdots + 2|y_1-z_1| + 2|y_0-z_0|\big)\\
    &\leq M (2m+1) \alpha\bigg(\max_{i} |y_i - z_i|\bigg).
\end{align*}
Therefore, if we have $m+1$ distinct points
$X = \{x_0,\dots,x_{m}\} \subset I$
with $P = P(X;f)$
then, according to the definition of $P$,
we have
for any $x \in [\min X,\max X]$ with $x \neq x_0$ that
\begin{align*}
|f(x)-P(x)| 
&= \left|f[x_0,\dots,x_m,x] (x-x_0)\cdots (x-x_m)\right|\\
&=\frac{|f[x_1,\dots,x_m,x] - f[x_0,\dots,x_m]|}{|x-x_0|}|x-x_0|\cdots |x-x_m|\\
&=\left| f[x_1,\dots,x_m,x] - f[x_0,\dots,x_m] \right||x-x_1|\cdots |x-x_m|\\
&\leq M (2m+1) \alpha(\diam(X)) \diam(X)^m.
\end{align*}
Since $f(x_0) = P(x_0)$, this gives the first inequality in \eqref{e-poly2}.

Now, for every $x \in [\min X, \max X]$
with $x \neq x_0$ and $x \neq x_1$,
Problem~7 on page 255 of \cite{NumAnalBook} 
and the symmetry of divided differences
implies that
\begin{align*}
    \frac{d}{dx} f[x_0,\dots,x_m,x]
    &=
    f[x_0,\dots,x_m,x,x]
    =
    \frac{f[x_1,\dots,x_m,x,x] - f[x_0,\dots,x_m,x]}{x-x_0}\\
    &=
    \frac{f[x_2,\dots,x_m,x,x] - f[x_1,\dots,x_m,x]}{(x-x_0)(x-x_1)} \\
    & \qquad \qquad-
    \frac{f[x_0,x_2\dots,x_m,x] - f[x_1,x_0,x_2,\dots,x_m]}{(x-x_0)(x-x_1)}.
\end{align*}
Thus, as above,
\begin{align*}
|f'(x)-P'(x)| 
&\leq \left|\frac{d}{dx} f[x_0,\dots,x_m,x] \cdot 
(x-x_0)\cdots (x-x_m)\right|\\
& \qquad +
|f[x_0,\dots,x_m,x]|\sum_{i=0}^m \prod_{j \neq i} |x-x_j|
\\
&\leq 
M (2m+1) (m+3) \alpha(\diam(X)) \diam(X)^{m-1}.
\end{align*}
The continuity of $f'$ and $P'$ gives the second inequality in \eqref{e-poly2}.
\end{proof}

\section{The $A/V$ condition}
\label{s-AV}
The following quantities were first defined in \cite{ZimSpePinWhitney}
to establish Theorem~\ref{t-HeisWhit}.
\begin{definition}
Suppose 
$F = (F^{k})_{k=0}^m$ and $G = (G^{k})_{k=0}^m$ are collections of continuous, real valued functions on a set $E \subseteq \R$,
and suppose $h:E \to \R$ is continuous.
Set $\gamma=(f,g,h) := (F^0,G^0,h)$.
For each $a,b \in E$, define the {\em area discrepancy} $A(\gamma;a,b)$ and {\em velocity} $V(\g ;a,b)$ as follows 
\begin{align*}
A(\g;a,b) &= h(b) - h(a) - 2 \int_a^b ((T_a F)'T_a G - (T_a G)'T_aF)\\
& \hspace{1in} +2f(a)(g(b) - T_a G(b)) - 2g(a)(f(b) - T_aF(b))\\
V(\g;a,b) &= (b-a)^{2m} + (b-a)^m \int_a^b \left( |(T_aF)'|+ |(T_aG)'| \right)
\end{align*}
\end{definition}
If $\gamma \in C^m(\R,\R^3)$,
we use the collections $F = (f^{(k)})_{k=0}^m$ and $G = (g^{(k)})_{k=0}^m$ in this definition as before unless otherwise noted.

\begin{definition}
Suppose $F = (F^{k})_{k=0}^m$ and $G = (G^{k})_{k=0}^m$ are collections of continuous, real valued functions on a set $E \subseteq \R$,
and suppose $h:E \to \R$ is continuous.
Set $\gamma=(f,g,h) := (F^0,G^0,h)$.
We
say that $\gamma$ satisfies the {\em $A/V$ condition on $E$}
if,
for all $\varepsilon > 0$, there is a $\delta > 0$ such that
$$
\frac{A(\g;a,b)}{V(\g;a,b)} < \varepsilon
\quad
\text{for all } a,b \in E \text{ with } |b-a|<\delta
.
$$
\end{definition}
Note that such a map $\gamma$ satisfies the $A/V$ condition vacuously on every finite set $E$.
As seen above, the $A/V$ condition is part of the necessary and sufficient conditions to guarantee the existence of smooth, horizontal extensions in Theorem~\ref{t-HeisWhit}.

We will now make a few observations about the quantities $A$ and $V$.
The following shows that they are left invariant with respect to the group operation on $\H$.
In particular, it will allow us to assume without loss of generality that $\gamma(a)=0$
when working with $A$ and $V$.
\begin{lemma}
\label{l-AVleftinv}
Suppose $\g \in C^m(\R,\R^3)$.
For any $p \in \H$ and $a,b \in \R$,
we have 
$$
A(p * \g;a,b) = A(\g;a,b) \quad \text{and} \quad V(p * \g;a,b) = V(\g;a,b).
$$
\end{lemma}
\begin{proof}
Fix $a,b, \in \R$ and $p \in \H$.
Write $\gamma = (f,g,h)$ and $p = (x,y,z)$,
and write $\hat{\gamma} = (\hat{f},\hat{g},\hat{h}) = p * \gamma$.
We then have
$$
\hat{f} = x + f,
\qquad
\hat{g} = y + g,
\qquad
\hat{h} = z + h + 2(yf-xg).
$$
Notice also that $T_a\hat{f} = T_af +x$ and $T_a \hat{g} = T_ag +y$.
Clearly, $V(p*\g,a,b) = V(\g,a,b)$.

Assume first that $\gamma(a) = 0$.
In this case, $\hat{\g}(a) = p$, so
\begin{align*}
A(p * \g;a,b) &= \hat{h}(b) - \hat{h}(a) - 2 \int_a^b \left((T_a \hat{f})'T_a \hat{g} - (T_a \hat{g})'T_a\hat{f}\right)\\
& \hspace{1in} +2\hat{f}(a)(\hat{g}(b) - T_a \hat{g}(b)) - 2\hat{g}(a)(\hat{f}(b) - T_a\hat{f}(b))\\
&= h(b) + 2(yf(b) - xg(b)) - 2 \int_a^b \left((T_a f)'(T_a g + y) - (T_a g)'(T_af+x)\right)\\
& \hspace{1in} +2x(g(b) - T_a g(b)) - 2y(f(b) - T_af(b))\\
&= h(b) - 2 \int_a^b \left((T_a f)'T_a g - (T_a g)'T_af\right)\\
& \hspace{1in} 
-2y \int_a^b (T_a f)' + 2x \int_a^b (T_a g)'
- 2xT_a g(b) + 2y T_af(b)\\
&= A(\g;a,b).
\end{align*}

If $\g(a)$ is arbitrary, then, since $\tilde{\g} := \g(a)^{-1} * \g$ satisfies $\tilde{\g}(a)=0$, we have from above
\begin{align*}
    A(\g;a,b) = A\left(\g(a)*\tilde{\g};a,b\right)
    =A\left(\tilde{\g};a,b\right)
    = A\left((p * \g(a)) * \tilde{\g};a,b\right)
    =A(p*\g;a,b).
\end{align*}
\end{proof}

When $\g$ is smooth, we are allowed to swap $a$ and $b$ in $A(\g;a,b)$ if we account for a small error term.

\begin{lemma}
\label{l-AVswap}
Suppose $\g \in C^m(\R,\R^3)$ and $K \subseteq \R$ is compact.
Assume $\alpha$ is a modulus of continuity 
such that $f$ and $g$ satisfy 
\eqref{e-taylor0}--\eqref{e-taylor2}
for all $x$ and $y$ in an interval containing $K$.
Then, for any $a,b \in K$,
we have 
$$
|A(\g;b,a)| \lesssim_{\g,K} |A(\g;a,b)| + \alpha(|b-a|)|b-a|^m.
$$
\end{lemma}
\begin{proof}
According to the previous lemma, we may assume that $\g(a)=0$.
Thus 
\begin{align}
    | A&\left(\gamma;b,a\right)| \nonumber \\
    &=
    \left| - h(b) - 2\int_b^a \left((T_bf)'T_bg - (T_bg)'T_bf\right)
    - 2f(b) T_bg(a) +2g(b) T_bf(a)
    \right| \nonumber \\
    &=
    \left| - h(b) - 2\int_b^a \left((T_af)'T_ag - (T_ag)'T_af\right)
    \right| \label{e-firstline}\\
    &\qquad +
     2\left|\int_b^a \left((T_af)'T_ag - (T_ag)'T_af\right) - \left((T_bf)'T_bg - (T_bg)'T_bf\right)\right|
     \label{e-secondline}\\
    &\qquad + 2\left|f(b) T_bg(a) - g(b) T_bf(a) \right| \nonumber.
\end{align}
Note that \eqref{e-firstline} is exactly $|A(\g;a,b)|$.
Also, for any $x$ between $a$ and $b$, we have
\begin{align*}
    \left|T_ag(x) - T_bg(x) \right| 
    &\leq 
    \left|T_ag(x) - g(x) \right| + \left|g(x) - T_bg(x) \right| \\
    &\leq \alpha(|x-b|)|x-b|^m + \alpha(|x-a|)|x-a|^m\\
    &\leq
    2\alpha(|b-a|)|b-a|^m.
\end{align*}
A similar argument gives $\left|(T_af)'(x) - (T_bf)'(x) \right| \leq 2 \alpha(|b-a|)|b-a|^{m-1}$ so that
\begin{align*}
\left|(T_af)'T_ag - (T_bf)'T_bg\right| 
&\leq |T_ag|\left|(T_af)' - (T_bf)'\right| + \left|(T_bf)'\right| \left|T_ag - T_bg\right| \\
&\lesssim_{f,g,K} \alpha(|b-a|)|b-a|^{m-1}.
\end{align*}
Swapping $f$ and $g$ and arguing similarly, we bound \eqref{e-secondline} by a constant multiple of $\alpha(|b-a|)|b-a|^m$.
Moreover, since $\g(a)=0$,
\begin{align*}
    |f(b) T_bg(a) - g(b) T_bf(a)|
    &\leq
    |f(b)| |T_bg(a)-g(a)| + |g(b)||T_bf(a) - f(a)|\\
    &\lesssim_{f,g,K} \alpha(|b-a|) |b-a|^m.
\end{align*}
This proves the lemma.
\end{proof}

\subsection{$A/V$ and horizontality}
The following is possibly the most useful observation from this paper.
As long as a $C^m$ curve satisfies the $A/V$ condition on a compact set $K \subseteq \R$,
the following lemma ensures that
we may drop the horizontality assumption (condition {\em (2)} in Theorem~\ref{t-HeisWhit}) on $K$.

\begin{lemma}
\label{l-horiz}
Suppose $f, g, h \in C^m(\R)$
and $K \subseteq \R$ is compact. 
If $\gamma =(f,g,h)$ 
satisfies the $A/V$ condition on $K$, 
then there is some $\hat{h} \in C^m(\mathbb{R})$
such that $\hat{h} = h$ on $K$, 
and
\begin{equation}
    \label{e-horiz1}
    \hat{h}' =  2(f'g-fg') \quad \text{ on } K.
\end{equation}
\end{lemma}
\begin{proof}
We will prepare our setting so that we may apply Whitney's classical extension theorem.
Set $H^0 := h$ on $K$,
and, for $1 \leq k \leq m$,
define $H^k = \eta^{(k-1)}|_K$
where $\eta \in C^{m-1}(\R)$ is defined as
\begin{align*}
    \eta = 2(f'g-g'f) \quad \text{ on } \R.
\end{align*}

\noindent \textbf{Claim:}
$H$ is a Whitney field of class $C^m$ on $K$.

Fix $a,b \in K$.
We will first check that 
$|h(b) - T_aH(b)|$
is uniformly $o(|b-a|^m)$ on $K$.
Using Lemma~\ref{l-AVleftinv} and the fact that the group operation in $\H$ is $C^\infty$ smooth,
we may assume that $\gamma(a)=0$.
Recalling the definition \eqref{e-TaylorJet} of the Taylor polynomial of a collection
and that $H^k(a) = \eta^{(k-1)}(a)$ for each $1 \leq k \leq m$,
observe that
\begin{align*}
    T_a H (b) 
    = \int_a^b (T_a^m H)'
    = \int_a^b T_a^{m-1} H^1 
    = \int_a^b T_a^{m-1} \eta
    = 2\int_a^b T_a^{m-1} (f'g) - T_a^{m-1}(fg').
\end{align*}
Here, we used the convention that $T_a^0H^1 = H^1(a)$.
Then
\begin{align*}
    |h(b) - T_aH(b)| 
    &\leq \left| h(b) - 2\int_a^b (T_a^mf)'T_a^mg - T_a^mf(T_a^mg)'\right| \\
    &\quad + 
    2 \int_a^b \left| (T_a^mf)'T_a^mg - T_a^{m-1} (f'g) \right|
    +2 \int_a^b \left|
    T_a^{m-1} (fg')  - T_a^mfT_a^{m-1}(g')\right|.
\end{align*}
Notice that, for any $x$ between $a$ and $b$,
\begin{align*}
    (T_a^mf)'(x)&T_a^mg(x) - T_a^{m-1} (f'g) (x)\\   
    &=
    \left[ T_a^{m-1}(f')(x) T_a^{m-1}g(x) - T_a^{m-1} (f'g)(x) \right]
    + (T_a^mf)'(x) \frac{g^{(m)}(x)}{m!}(x-a)^m,
\end{align*}
and recall that $T_a^{m-1} (f'g)(x)$
is simply the polynomial 
consisting
of those terms in the polynomial $T_a^{m-1}(f')(x) \cdot T_a^{m-1}g(x)$ which have degree at most $m-1$.
Therefore, the quantity in brackets above
is a polynomial in $(x-a)$ whose 
coefficients are 
linear combinations of derivatives of $f$ and $g$
and whose terms have degree at least $m$.
Thus there is a constant $C > 0$ depending only on the derivatives of $f$ and $g$ on $K$ such that
\begin{align*}
    \left|(T_a^mf)'T_a^mg - T_a^{m-1} (f'g)\right|
    \leq C
    |b-a|^m
\end{align*}
between $a$ and $b$. Swapping $f$ and $g$ and repeating the above discussion, we find that
\begin{align}
    |h(b) - T_aH(b)| 
    &\leq 
    \left| h(b) - 2\int_a^b (T_a^mf)'T_a^mg - (T_a^mg)'T_a^mf\right|
    +
    4 C
    |b-a|^{m+1} \nonumber\\
    &= \left| A\left(\g;a,b\right) \right|
    + 4 C
    |b-a|^{m+1}.\label{e-Aswap}
\end{align}
If $a \leq b$, then
$|h(b) - T_aH(b)|$ is uniformly $o(|b-a|^m)$ on $K$
since $\gamma$ satisfies the $A/V$ condition on $K$.
If $a > b$, we choose a modulus of continuity $\alpha$
such that $f$ and $g$ satisfy 
\eqref{e-taylor0}--\eqref{e-taylor2} on an interval containing $K$,
we apply Lemma~\ref{l-AVswap} to \eqref{e-Aswap}, and then we apply the previous sentence.

It remains to check that $|H^k(b) - T_a^{m-k} H^k (b)|$
is uniformly $o(|b-a|^{m-k})$ on $K$ for $1 \leq k \leq m$,
but this follows easily from the definition of $H^k$ since, for such a $k$,
$$
|H^k(b) - T_a^{m-k} H^k (b)|
=
|\eta^{(k-1)}(b) - T_a^{m-k} (\eta^{(k-1)}) (b)|
$$
which is uniformly $o(|b-a|^{(m-1)-(k-1)})$ on $K$
since $\eta$ is of class $C^{m-1}$.
This proves the claim.

Therefore, according to Whitney's classical extension theorem
(Theorem~\ref{t-WhitClassic}),
there is a $C^m$ extension $\hat{h}$ of $H$.
In particular, we have $\hat{h}(x) = H^0(x) = h(x)$ for all $x \in K$,
and, by the definition of $H$,
$$
\hat{h}' =  H^1 = \eta = 2(f'g-g'f) \quad \text{on } K.
$$
This completes the proof of the lemma.
\end{proof}

We will now record a version of the above result for $C^{m,\omega}$ curves
to be used in the proof of Theorem~\ref{t-finiteness}.

\begin{lemma}
\label{l-horizLip}
Suppose $f, g, h \in C^{m,\omega}(\R)$
and $K \subseteq \R$ is compact. 
If $\gamma =(f,g,h)$ 
satisfies the $A/V$ condition on $K$, 
then there is some $\hat{h} \in C^{m,\omega}(\mathbb{R})$
such that $\hat{h}|_K = h$, 
and
$$
    \hat{h}' =  2(f'g-fg') \quad \text{ on } K.
$$
\end{lemma}
\begin{proof}
The proof of this lemma is nearly identical to the previous one.
The main difference is that we must use the fact that $\eta$ is now of class $C^{m-1,\omega}$
to conclude that
$$
|H^k(b) - T_a^{m-k} H^k (b)|
=
|\eta^{(k-1)}(b) - T_a^{m-k} (\eta^{(k-1)}) (b)| 
\leq C \omega(|b-a|) |b-a|^{m-k} 
$$
for some constant $C>0$.
Thus, we may apply Theorem~\ref{t-WhitClassicLip} in lieu of Theorem~\ref{t-WhitClassic} to construct a $C^{m,\omega}$ extension $\hat{h}$ of $h$ 
and conclude the lemma.
\end{proof}

\section{The discrete $A/V$ condition}
\label{s-dAV}
\begin{definition}
Fix $E \subseteq \R$ and $\gamma=(f,g,h):E \to \mathbb{R}^3$.
Suppose $X \subseteq E$ with $\# X = m+1$, 
and 
set
$P_f = P(X;f)$ and $P_g = P(X;g)$.
For any $a,b \in X$, define the {\em discrete area discrepancy} $A[X,\g;a,b]$ 
and {\em discrete velocity} $V[X,\g;a,b]$ as follows:
\begin{align*}
A[X,\g;a,b] &= h(b) - h(a) - 2 \int_{a}^{b} (P_f'P_g - P_g'P_f)\\
V[X,\g;a,b] &= \diam(X)^{2m} + \diam(X)^m \int_a^b \left(|P_f'|+ |P_g'|\right).
\end{align*}
\end{definition}
Note in particular that the definitions of $A[X,\g;a,b]$ and $V[X,\g;a,b]$ depend only on the functions $f$, $g$, and $h$ rather than a family of functions (as the definitions of $A(\g;a,b)$ and $V(\g;a,b)$ do).

\begin{definition}
For any set $E \subseteq \mathbb{R}$ with and $\gamma:E \to \mathbb{R}^3$,
say that $\gamma$ satisfies the {\em discrete $A/V$ condition on $E$}
if, for every $\varepsilon > 0$, there is some $\delta > 0$ such that
$$
\frac{A[X,\g ; a,b]}{V[X,\g ; a,b]} < \varepsilon
$$
for any finite set $X \subseteq E$ with $\# X = m+1$ and $\diam X < \delta$ and any $a,b \in X$ with $a<b$.
\end{definition}

Note again that such a map $\gamma$ satisfies the discrete $A/V$ condition vacuously on every finite set $E$.
We once again have a left invariance property for the discrete versions of $A$ and $V$.
\begin{lemma}
\label{l-AVleftinvDisc}
Suppose $\g :K \to \H$ for some $K \subseteq \R$,
and suppose $X \subseteq K$ with $\# X = m+1$.
Fix $a,b \in X$ and $p \in \H$.
Then 
$$
A[X,p *\g;a,b] = A[X,\g;a,b]
\quad
\text{and}
\quad
V[X,p * \g;a,b] = V[X,\g;a,b].
$$
\end{lemma}
\begin{proof}
This lemma follows almost
exactly the same proof as that of Lemma~\ref{l-AVleftinv}
since
$P_{\hat{f}} = P_f + x$ and $P_{\hat{g}} = P_g + y$
and since $P_f(x) = f(x)$ and $P_g(x) = g(x)$ for any $x \in X$ by the definition of the interpolating polynomials.
\end{proof}

\subsection{Equivalence of the conditions for $C^m$ curves}
Here, we will compare the $A/V$ and discrete $A/V$ conditions for $C^m$ curves.

\begin{lemma}
\label{l-AV}
Suppose $\gamma \in C^m(\R,\mathbb{R}^3)$ and $K \subseteq \R$ is a compact set containing at least $m+1$ points.
If $\g$ satisfies the $A/V$ condition on $K$,
then, for every $\varepsilon > 0$, there is some $\delta > 0$ such that
$$
\left| \frac{A(\g;a,b)}{V(\g;a,b)} - \frac{A[X,\g ; a,b]}{V[X,\g ; a,b]}\right| < \varepsilon
$$
for any finite set $X \subseteq K$ with $\# X = m+1$ and $\diam X < \delta$ and any $a,b \in X$ with $a<b$.
In particular, if $\gamma$ satisfies the $A/V$ condition on $K$, then 
$\g$ satisfies the discrete $A/V$ condition on $K$.
\end{lemma}
\begin{proof}
Write $\g = (f,g,h)$.
We may choose a modulus of continuity $\alpha$ and a constant $C>0$
such that
$f$ and $g$ satisfy 
\eqref{e-poly2}
for any
$X \subset I$
with $\#X = m+1$
and
\eqref{e-taylor0}--\eqref{e-taylor2}
for all $x,y \in I$
where $I=[ \min K, \max K]$.

Suppose $X$ is a set of $m+1$ distinct 
points in $K$. Choose $a,b \in X$ with $a <b$.
By Lemmas~\ref{l-AVleftinv} and \ref{l-AVleftinvDisc},
we may assume without loss of generality that $\g(a) = 0$.
For simplicity, 
write $A = A(\g;a,b)$ and $V=V(\g;a,b)$,
and write
$A_X = A[X,\g;a,b]$ and $V_X = V[X,\g;a,b]$.
Then
\begin{equation}
\left| \frac{A}{V} - \frac{A_X}{V_X}\right|
\leq
\left| \frac{A}{V} - \frac{A}{V_X}\right|
+
\left| \frac{A - A_X}{V_X}\right|
=
\left| \frac{A}{V} \right|
\left| \frac{V_X - V}{V_X}\right|
+
\left| \frac{A - A_X}{V_X}\right|.
\label{e-AV}
\end{equation}
Write $\alpha := \alpha(b-a)$.
We will prove that $|(A-A_X)/V_X|$
is bounded by a constant multiple of $\alpha$
and that $|(V_X-V)/V_X|$ is bounded.

To begin, notice that
\begin{align}
\label{e-AminusA}
A - A_X 
= 
2 \int_a^b \left[ (P_f'P_g  - (T_af)'T_ag) + ((T_ag)'T_af - P_g'P_f) \right].
\end{align}
Let us bound the first term in this integrand. Note that
\begin{align*}
P_f'P_g  - (T_af)'T_ag
=
P_f'(P_g - T_ag)
&+P_g(P_f'  - (T_af)')
\\
&\quad +((T_af)' - P_f')(P_g-T_ag).
\end{align*}
Now
Lemma~\ref{l-poly} gives
\begin{align}
\label{e-1st}
|P_g  - T_ag|
\leq
|P_g - g| + |g - T_ag|
\leq
(C+1) \alpha \diam(X)^m
\end{align}
and
\begin{align}
\label{e-2nd}
|P_f' - (T_af)'|
\leq
|P_f' - f'| + |f' - (T_af)'|
\leq
(C+1)\alpha \diam(X)^{m-1}
\end{align}
on $[a,b]$.
Therefore, 
\begin{align}
\label{e-3rd}
|(T_af)' - P_f'||P_g-T_ag| \leq (C+1)^2 \alpha^2 \diam(X)^{2m-1}
\end{align}
on $[a,b]$.
Combining \eqref{e-1st}, \eqref{e-2nd}, and \eqref{e-3rd} gives
\begin{align*}
    \int_a^b |P_f'P_g  - (T_af)'T_ag|
    \lesssim_C 
    \alpha \diam(X)^m \int_a^b |P_f'|
    &+
    \alpha \diam(X)^{m-1} \int_a^b |P_g|\\
    &\quad +
    \alpha^2 \diam(X)^{2m}
\end{align*}
According to Corollary~2.11 in \cite{ZimSpePinWhitney}
applied to the polynomial $P_g$,
we have
$$
\int_a^b |P_g| 
\leq 8m^2(b-a)\int_a^b |P_g'|.
$$
Hence 
\begin{align*}
    \int_a^b |P_f'P_g  - (T_af)'T_ag|
    \lesssim_C
     \alpha^2 \diam(X)^{2m}
     +
     (1+8m^2)\alpha \diam(X)^m \int_a^b |P_f'| + |P_g'|.
\end{align*}
Similar arguments give
\begin{align*}
    \int_a^b |(T_ag)'T_af - P_g'P_f|
    \lesssim_C
    \alpha^2 \diam(X)^{2m}
     +
    (1+8m^2)\alpha \diam(X)^m \int_a^b |P_f'| + |P_g'|,
\end{align*}
and inputting these bounds into \eqref{e-AminusA} gives
\begin{align*}
|A - A_X |
&\lesssim_{m,C}
\alpha^2 \diam(X)^{2m}
     +
    \alpha \diam(X)^m \int_a^b |P_f'| + |P_g'|\\
    &\lesssim_{m,C}
(\alpha^2
     +
    \alpha) V_X.
\end{align*}
This bounds the second term in \eqref{e-AV}.
To bound the first term, notice that
\begin{align*}
    |V_X - V| 
    &\leq 
    \diam(X)^{2m}
    +
    (b-a)^{2m}
    +
    \left|\diam(X)^m - (b-a)^m\right|
    \int_a^b\left(|P_f'| + |P_g'|\right)\\
    & \hspace{1.5in}
    +
    (b-a)^m\int_a^b\left||P_f'| - |(T_af)'| + |P_g'| - |(T_ag)'|\right|.
\end{align*}
As above, we have
\begin{align*}
    \left||P_f'| - |(T_af)'|\right| \leq |P_f' - (T_af)'|
    \leq(C+1)
    \alpha \diam(X)^{m-1}
\end{align*}
and
\begin{align*}
    \left||P_g'| - |(T_ag)'|\right| \leq |P_g' - (T_ag)'|
    \leq(C+1)
    \alpha \diam(X)^{m-1}.
\end{align*}
Hence, the bound $\alpha(b-a) \leq \alpha(\diam K)$ gives
\begin{align*}
    |V_X - V|
    \lesssim_{C,\alpha,K}
    \diam(X)^{2m}
    +
    \diam(X)^{m} \int_a^b\left(|P_f'| + |P_g'|\right)
    = 
    V_X.
\end{align*}
Thus, by \eqref{e-AV},
$$
\left|\frac{A}{V} - \frac{A_X}{V_X} \right|
\lesssim_{m,C,\alpha,K} \left|\frac{A}{V}\right| + \alpha.
$$
Since $\g$ satisfies the $A/V$ condition on $K$, the proof is complete.
\end{proof}

\begin{lemma}
\label{l-AV2}
Suppose $\gamma \in C^m(\R,\mathbb{R}^3)$ 
and $K \subseteq \mathbb{R}$ is a compact set containing at least $m+1$ points
with finitely many isolated points.
If $\gamma$ satisfies the discrete $A/V$ condition on $K$, then 
it satisfies the $A/V$ condition on $K$.

Moreover, if $m=1$, then $K$ can be any compact set containing at least 2 points.
\end{lemma}
\begin{proof}
If $K$ is finite, the definition of the $A/V$ condition is vacuously true.
Otherwise, choose 
a modulus of continuity $\alpha$
as in the proof of Lemma~\ref{l-AV}
which also satisfies
\begin{equation}
\label{e-AV-discrete-ass}    
\frac{A[X,\g ; a,b]}{V[X,\g ; a,b]} \leq  \alpha(\diam(X))
\end{equation}
for all $a,b \in X \subseteq K$ with $a<b$ and $\# X = m+1$.

Let $a,b \in K$ with $a<b$.
By our assumption on $K$, 
we may assume that either $a$ or $b$ is a limit point of $K$.
Thus we can choose a finite set $X$ consisting of $a$, $b$, and $m-1$ other distinct points in $K$ within $(b-a)/2$ of $a$ or $b$.
In particular, $\diam (X) \leq 2(b-a)$.
(When $m=1$, we may skip this argument since $X = \{a,b\}$, and so $\diam(X)=b-a$.
In particular, there is no need to concern ourselves with limit points or isolated points in this case.)

The proof of this lemma will now be nearly identical to that of the previous lemma. 
We will only note the main differences.
Write $A$, $V$, $A_X$, $V_X$, $\alpha$, and $C$ as before.
Also, as before (but slightly differently), write
\begin{align}
\left| \frac{A}{V} - \frac{A_X}{V_X}\right|
\leq
\left| \frac{A_X}{V_X} \right|
\left| \frac{V - V_X}{V}\right|
+
\left| \frac{A - A_X}{V}\right|.
\label{e-AV2}
\end{align}
Again, Lemmas~\ref{l-AVleftinv} and \ref{l-AVleftinvDisc} ensure that
we may assume $\g(a) = 0$.
With \eqref{e-AminusA} in mind, write
\begin{align*}
P_f'P_g  - (T_af)'T_ag
=
(T_af)'(P_g - T_ag)
&+T_ag(P_f'  - (T_af)')
\\
&\quad +(P_f' - (T_af)')(P_g-T_ag).
\end{align*}
By applying Corollary~2.11 in \cite{ZimSpePinWhitney}
to the polynomials $T_ag$ and $T_af$,
we may use \eqref{e-1st}, \eqref{e-2nd}, and \eqref{e-3rd} and 
the fact that $\diam(X) \leq 2(b-a)$
to conclude as before that
\begin{align*}
|A - A_X|
&\lesssim_{m,C}
\alpha^2 (b-a)^{2m}
     +
    \alpha(b-a)^m \int_a^b |T_af'| + |T_ag'|\\
    &\lesssim_{m,C}
(\alpha^2
     +
    \alpha) V.
\end{align*}
Moreover, arguing as above using the fact that $\diam(X) \leq 2(b-a)$, we have
\begin{align*}
    |V - V_X| 
    &\leq 
    (b-a)^{2m}
    +
    \diam(X)^{2m}
    +
    \left|(b-a)^m - \diam(X)^m\right|
    \int_a^b\left(|(T_af)'| + |(T_ag)'|\right)\\
    & \hspace{1.5in}
    +
    \diam(X)^m\int_a^b\left||(T_af)'| - |P_f'|+ |(T_ag)'| - |P_g'| \right|\\
    &\lesssim_{C,m,\alpha,K}
    (b-a)^{2m} + (b-a)^m
    \int_a^b\left(|(T_af)'| + |(T_ag)'|\right) = V.
\end{align*}
By \eqref{e-AV-discrete-ass}, the proof is complete.
\end{proof}

\subsection{Stronger equivalence for $C^{m,\omega}$ curves}

We now record two analogous results which will be important in the proof of the finiteness principle Theorem~\ref{t-finiteness}.
The additional regularity on $|A/V|$ provides more control on the difference between the $A/V$ fractions.

\begin{lemma}
\label{l-AVLip}
Suppose $\gamma \in C^{m,\omega}(\R,\mathbb{R}^3)$, and suppose $K \subseteq \R$ is compact with $\#K \geq m+1$.
Assume that there is a constant $M>0$ such that $|A(\g;a,b)/V(\g;a,b)| \leq M \omega(b-a)$ for all $a,b \in K$ with $a<b$.
Then for any $X \subseteq K$ with $\#X = m+1$ and $a,b \in X$ with $a<b$, we have
$$
\left| \frac{A(\g;a,b)}{V(\g;a,b)} - \frac{A[X,\g ; a,b]}{V[X,\g ; a,b]}\right| 
\leq C_0
\omega(\diam(X))
$$
where $C_0 \geq 1$ is a polynomial combination of $m$, $M$, $\Vert \gamma \Vert_{C^{m,\omega}(\R,\mathbb{R}^3)}$, $\omega(\diam K)$, and $\diam(K)$.
\end{lemma}
\begin{proof}
The proof of this lemma is then nearly identical to 
that of Lemma~\ref{l-AV}.
Indeed, we may simply replace all instances of $\alpha$ with a constant multiple of $\omega$. 
Our added assumption that $|A/V| \leq M \omega(b-a) \leq M \omega(\diam(X))$
completes the proof.
\end{proof}

The proof of the final lemma follows from the proof of Lemma~\ref{l-AV2} in the same way that the proof of Lemma~\ref{l-AVLip} followed from the proof of Lemma~\ref{l-AV}.

\begin{lemma}
\label{l-AVLip2}
Suppose $\gamma \in C^{m,\omega}(\R,\mathbb{R}^3)$, and suppose $K \subseteq \R$ is compact with finitely many isolated points and $\#K \geq m+1$.
Assume that there is a constant
$M>0$ such that
$|A[Y,\g;a,b]/V[Y,\g;a,b]| \leq M \omega(\diam(Y))$
for all $Y \subseteq K$ with $\# Y = m+1$ and $a,b \in Y$ with $a<b$. 

For any $a,b \in K$ with $a<b$, 
there is a set $X \subseteq K$ containing $a$ and $b$ with $\# X = m+1$ and $\diam(X) \leq 2(b-a)$ 
such that
$$
\left| \frac{A(\g;a,b)}{V(\g;a,b)} - \frac{A[X,\g ; a,b]}{V[X,\g ; a,b]}\right| 
\leq C_1
\omega(b-a)
$$
where $C_1 \geq 1$ is a polynomial combination of $m$, $M$, $\Vert \gamma \Vert_{C^{m,\omega}(\R,\mathbb{R}^3)}$, $\omega(\diam K)$, and $\diam(K)$.
\end{lemma}

\section{Proofs of the main theorems}
\label{s-proofs}

In this section, we will prove Theorems~\ref{c-m1}, \ref{t-supermain}, \ref{t-supermainWf}, and \ref{t-finiteness}.

\subsection{Answering Whitney's question in $\H$ for $n =1$}
We will first observe that the assumptions of Theorems~\ref{c-m1} and \ref{t-supermain} are necessary
even when $K$ is an arbitrary compact set.

\begin{proposition}
If
$\gamma \in C^m_\mathbb{H}(\mathbb{R})$
and $K \subseteq \mathbb{R}$ is compact,
then
\begin{enumerate}
\item \label{nec1}
the $m$th divided differences of $\gamma$ converge uniformly in $K$, and
\item \label{nec3} $\gamma$ satisfies the discrete $A/V$ condition on $K$.
\end{enumerate}
\end{proposition}

\begin{proof}
Condition {\em (\ref{nec1})} follows from Theorem~\ref{t-WhitFin},
and Theorem~\ref{t-HeisWhit} implies that $\gamma$ satisfies the $A/V$ condition on $K$.
Thus Lemma~\ref{l-AV} gives {\em (\ref{nec3})}.
\end{proof}

We will now prove sufficiency in Theorem~\ref{t-supermain}.
Our restriction on the set $K$ appears here 
only because we will be applying Lemma~\ref{l-AV2}.

\begin{theorem}
\label{t-backward}
Suppose $K \subseteq \R$ is compact with finitely many isolated points, and fix $\gamma :K \to \mathbb{H}$.
Assume
\begin{enumerate}
\item
the $m$th divided differences of $\gamma$ converge uniformly in $K$, and
\item $\gamma$ satisfies the discrete $A/V$ condition on $K$.
\end{enumerate}
Then there is a horizontal curve $\Gamma \in C^m(\R,\R^3)$ such that $\Gamma|_K = \gamma$.
\end{theorem}
\begin{proof}
We may clearly assume that $K$ contains at least $m+1$ points.
By {\em (1)} and Theorem~\ref{t-WhitFin},
there is some $\tilde{\g} =(f,g,h) \in C^m(\R,\R^3)$ such that $\tilde{\g}|_K = \g$. 
By {\em (2)} and Lemma~\ref{l-AV2},
$\tilde{\gamma}$ satisfies the $A/V$ condition on $K$.
Therefore, Lemma~\ref{l-horiz}
implies that there is some $\hat{h} \in C^m(\mathbb{R})$
such that $\hat{h}|_K = h|_K$,
and
    $\hat{h}' =  2(f'g-fg')$ on $K$.
By the Leibniz rule, then, we have
\begin{equation*}
\hat{h}^{(k)} = 2 \sum_{i=0}^{k-1} \binom{k-1}{i}
\left(f^{(k-i)}g^{(i)} - g^{(k-i)}f^{(i)}\right)
\quad
\text{ on } K.
\end{equation*}
    
Set $\hat{\g} = (f,g,\hat{h}):\R \to \mathbb{H}$.
By our above construction, we have that $\hat{\g}|_K = \tilde{\g}|_K=\g$
and that the collection $(F^k,G^k,H^k)_{k=0}^m = \left( \hat{\g}^{(k)} \right)_{k=0}^m$ on $K$ satisfies the assumptions of 
Theorem~\ref{t-HeisWhit}.
Therefore, there is a $C^m$ horizontal curve $\Gamma:\R \to \mathbb{H}$ such that $\Gamma|_K = \hat{\g}|_K = \g$.
\end{proof}

\subsection{The special case $m=1$}
Here we prove Theorem~\ref{c-m1}, which holds for arbitrary compact sets.
\begin{proof}
The necessity follows from Theorem~1.7 in \cite{ZimWhitney}.
Now fix a compact set $K \subseteq \R$ and
a continuous map $\gamma=(f,g,h):K \to \mathbb{R}^3$
such that the Pansu difference quotients of $\gamma$ converge uniformly on $K$ to horizontal points.
That is, the classical Euclidean difference quotients of $f$ and $g$ converge uniformly on $K$
and there is a modulus of continuity $\alpha$ such that
$$
|h(b) - h(a) - 2 (f(b)g(a) - f(a)g(b))| \leq \alpha(|b-a|) |b-a|^2
$$
for all $a,b \in K$.
This implies in particular that
\begin{align*}
    |h(b)-h(a)| &\leq |h(b) - h(a) - 2 (f(b)g(a) - f(a)g(b))|
    +
    2|f(b)g(a) - f(a)g(b)|\\
    &\leq \alpha(|b-a|)|b-a|^2
    + 2 |g(a)||f(b)-f(a)| + 2|f(a)||g(a)-g(b)|\\
    &\lesssim_{f,g} \alpha(|b-a|)\left(|b-a|^2 + |b-a|\right).
\end{align*}
That is, the difference quotients of $h$ converge uniformly on $K$ as well.
In other words, the 1st divided differences of $\gamma$ converge uniformly in $K$,
and so there is some $\tilde{\g} = (\tilde{f},\tilde{g},\tilde{h}) \in C^m(\R,\R^3)$ such that $\tilde{\g}|_K = \g$.

Now suppose $X=\{a,b\} \subseteq K$ with $a<b$, 
and 
set
$P_f = P(X;f)$ and $P_g = P(X;g)$.
Since $P_f$ and $P_g$ are affine,
we have
\begin{align*}
A[X,\g;a,b] 
&= h(b) - h(a) - 2 \int_{a}^{b} (P_f'P_g - P_g'P_f)\\
&= h(b) - h(a) - 2 (f(b)g(a) - f(a)g(b)),\\
V[X,\g;a,b] &= \diam(X)^{2} + \diam(X) \int_a^b \left|\tfrac{f(b)-f(a)}{b-a}\right|+ \left|\tfrac{g(b)-g(a)}{b-a}\right|
\geq \diam(X)^2
= (b-a)^2.
\end{align*}
Therefore, 
since $\tilde{\gamma} = \g$ on $K$,
$$
\left| \frac{A[X,\tilde{\g};a,b] }{V[X,\tilde{\g};a,b]}\right|
=
\left| \frac{A[X,\g;a,b] }{V[X,\g;a,b]}\right|
\leq 
\frac{|h(b) - h(a) - 2 (f(b)g(a) - f(a)g(b))|}{|b-a|^2}
\leq \alpha(|b-a|).
$$
Hence $\tilde{\g}$ satisfies the discrete $A/V$ condition on $K$.
It follows from Lemma~\ref{l-AV2} 
that
$\tilde{\g}$ satisfies the $A/V$ condition on $K$ as well since $m=1$.
Thus Lemma~\ref{l-horiz} gives
an $\hat{h} \in C^1(\R)$ such that 
$\hat{h}|_K = \tilde{h}|_K = h$
and $\hat{h}' = 2(\tilde{f}'\tilde{g}-\tilde{f}\tilde{g}')$ on $K$.
Setting $\hat{\g} = (\tilde{f},\tilde{g},\hat{h})$,
it follows that $\hat{\g}|_K$ and $\hat{\g}'|_K$ 
satisfy the assumptions of Theorem~1.7 in \cite{ZimWhitney}.
We may conclude that there is a horizontal curve $\Gamma \in C^1(\R,\R^3)$ such that $\Gamma|_K = \gamma$.
\end{proof}

\subsection{Answering Whitney's question in $\H$ using $W$}
Here we prove Theorem~\ref{t-supermainWf}.

\begin{proof}
Necessity once again follows from Theorems~\ref{t-WhitFin} and \ref{t-HeisWhit}.
Assume now that $K$ is compact, $\gamma:K \to \H$ is continuous, and the $m$th divided differences of $\g$ converge uniformly on $K$.
In particular, $W\g = (f,g,h)$ exists.
Assume $W\g$ satisfies the $A/V$ condition on $K$.
Lemma~\ref{l-horiz} then gives a function $\hat{h} \in C^m(\R)$ such that $\hat{h}|_K = h|_K$ and $\hat{h}' = 2(f'g-fg')$
on $K$.
Once again defining $\hat{\g} = (f,g,\hat{h})$,
the collection $\left( \hat{\g}^{(k)} \right)_{k=0}^m$ on $K$ satisfies the assumptions of 
Theorem~\ref{t-HeisWhit},
so there is a horizontal curve $\Gamma \in C^m(\R,\R^3)$ such that $\Gamma|_K = \hat{\g}|_K = (W\g)|_K = \g$.
\end{proof}

\subsection{The finiteness principle}
This subsection is devoted to the proof of Theorem~\ref{t-finiteness}.
\begin{proof}
Suppose $K \subseteq \R$ is compact with finitely many isolated points and $\#K \geq m+2$.
Suppose $M>0$ and $\gamma:K\to \R^3$ satisfy the following:
for any $X \subseteq K$ with $\#X = m+2$,
there is a function $\Gamma_X \in C^{m,\omega}(\R,\R^3)$
such that $\Gamma_X = \gamma$ on $X$,
$
\Vert \Gamma_X \Vert_{C^{m,\omega}(\R,\R^3)} \leq M,
$
and
\begin{equation}
\label{e-assump}    
\left| \frac{A(\Gamma_X;a,b)}{V(\Gamma_X;a,b)} \right| \leq M \omega(b-a)
\quad \text{ for all } a,b \in X \text{ with } a<b.
\end{equation}

Suppose $X=\{x_0,\dots,x_{m+1}\}$ is 
a set of $m+2$ distinct points in $K$
with $x_0 < x_1 < \cdots < x_{m+1}$.
Choose $\Gamma_X \in C^{m,\omega}(\R,\R^3)$ as above.
Then
\begin{align*}
    |\gamma[X]| 
    =|\Gamma_X[X]|
    = \frac{|\Gamma_X[x_1,\dots,x_{m+1}] - \Gamma_X[x_0,\dots,x_{m}]|}{|x_{m+1} - x_0|}
    &= \frac{|\Gamma_X^{(m)}(y_1) - \Gamma_X^{(m)}(y_2)|}{(m+1)!|x_{m+1} - x_0|}\\
    &\leq \frac{M}{(m+1)!} \frac{\omega(\diam(X))}{\diam(X)}
\end{align*}
for some $y_1 \in (x_1,x_{m+1})$ and $y_2 \in (x_0,x_{m})$
by the mean value theorem for divided differences.
According to Theorem~A in \cite{BruShv},
there is some $\tilde{\g} =(f,g,h) \in C^{m,\omega}(\R,\R^3)$ such that $\tilde{\g}|_K = \g$.

We will now observe that there is a constant $\tilde{C}>0$ 
depending only on $m$, $M$, $\omega(\diam K)$, and $\diam(K)$
such that $\tilde{\g}$ satisfies 
\begin{align}
\left| \frac{A[Y,\tilde{\g};a,b]}{V[Y,\tilde{\g};a,b]} \right| 
\leq \tilde{C} \omega(\diam(Y)) \label{e-discretechain}
\end{align}
for any $a,b \in Y \subseteq K$ with $a<b$ and $\# Y = m+1$.
Indeed, choose such a set $Y$, set $X = Y \cup \{ x \}$ for some $x \in K \setminus Y$,
and choose $\Gamma_X$ as above.
Since $\Gamma_X = \g = \tilde{\gamma}$ on $Y$,
it follows from 
\eqref{e-assump} and
Lemma~\ref{l-AVLip} that,
for any $a,b \in Y$ with $a<b$,
\begin{align*}
\left| \frac{A[Y,\tilde{\g};a,b]}{V[Y,\tilde{\g};a,b]} \right| 
=
\left| \frac{A[Y,\Gamma_X;a,b]}{V[Y,\Gamma_X;a,b]} \right| 
    &\leq
    \left| \frac{A[Y,\Gamma_X;a,b]}{V[Y,\Gamma_X;a,b]} - \frac{A(\Gamma_X;a,b)}{V(\Gamma_X;a,b)} \right|
    +
    \left| \frac{A(\Gamma_X;a,b)}{V(\Gamma_X;a,b)} \right| \nonumber\\
    &\lesssim_{m,M,\omega,K} \omega(\diam(Y)).
\end{align*}
In particular, $\tilde{\g}$ satisfies the discrete $A/V$ condition on $K$, and thus, 
by Lemma~\ref{l-AV2}, it satisfies the $A/V$ condition on $K$ as well.
Now, by Lemma~\ref{l-horizLip},
there is some
$\hat{h} \in C^{m,\omega}(\mathbb{R})$
such that $\hat{h}|_K = h|_K$ and
$\hat{h}' =  2(f'g-fg')$ on $K$.

Set $\hat{\g} = (f,g,\hat{h})$.
As before, the collection
$\left( \hat{\g}^{(k)} \right)_{k=0}^m$ on $K$
satisfies
conditions {\em (1)} and {\em (2)} from  Theorem~\ref{t-HeisWhitLip}.
We will now verify condition {\em (3)}.
Note that \eqref{e-discretechain} holds with $\hat{\g}$ in place of $\tilde{\g}$
since
$\hat{\g} = \tilde{\g}$ on $K$.
Fix $a,b \in K$ with $a<b$.
According to Lemma~\ref{l-AVLip2},
there is a set $Y \subseteq K$ containing $a$ and $b$ with $\# Y = m+1$
and $\diam(Y) \leq 2(b-a)$ 
such that 
\begin{align*}
\left| \frac{A(\hat{\g};a,b)}{V(\hat{\g};a,b)} \right| 
\leq
\left| \frac{A(\hat{\g};a,b)}{V(\hat{\g};a,b)} - \frac{A[Y,\hat{\g};a,b]}{V[Y,\hat{\g};a,b]} \right| 
+ \left| \frac{A[Y,\hat{\g};a,b]}{V[Y,\hat{\g};a,b]} \right| 
\leq \hat{C} \omega(b-a)
\end{align*}
where $\hat{C}>0$ is a constant depending only on $m$, $M$, $\omega$, $K$, and $\Vert \hat{\gamma} \Vert_{C^{m,\omega}(\R,\R^3)}$.
We may therefore apply Theorem~\ref{t-HeisWhitLip}
to find a horizontal curve $\Gamma \in C^{m,\sqrt{\omega}}(\R,\R^3)$ such that $\Gamma|_K = \hat{\g}|_K = \tilde{\g}|_K = \g$.

\end{proof}

\bibliography{zimbib} 
\bibliographystyle{acm}

\end{document}